\documentclass{article}
\usepackage{amsmath}
\usepackage{amsthm}
\usepackage{arxiv}
\usepackage[utf8]{inputenc} 
\usepackage[T1]{fontenc}    
\usepackage{hyperref}       
\usepackage{url}            
\usepackage{booktabs}       
\usepackage{amsfonts}       
\usepackage{nicefrac}       
\usepackage{microtype}      
\usepackage{cleveref}       
\usepackage{lipsum}         
\usepackage{graphicx}
\usepackage{doi}
\usepackage{paralist}
\usepackage{mathrsfs}
\usepackage{dsfont}
\usepackage{graphicx}
\usepackage{caption}
\usepackage[dvipsnames]{xcolor}
\usepackage{cite}
\usepackage[misc]{ifsym}
\usepackage{epsfig} 
\usepackage{epstopdf}

\newtheorem{theorem}{Theorem}[section]

\newtheorem{lemma}[theorem]{Lemma}

\newtheorem{remark}[theorem]{Remark}

\renewcommand*\div{\mathbf{div}} 

\renewcommand*\div{\mathbf{div}} 

\newcommand{\intgwed}[1]{\int_{\omega_{\frac{\varepsilon}{2}}} #1 \, dx}
\newcommand{\intgwe}[1]{\int_{\omega_\varepsilon} #1 \, dx}
\newcommand{\intgg}[1]{\int_{\Gamma_1} #1 \, dS}
\newcommand{\intgo}[1]{\int_\Omega #1 \, dx}
\newcommand{\intt}[1]{\int_0^T #1 \, dt}
\newcommand{\ints}[1]{\int_{s_1}^{s_2} #1 \, dt}

\newcommand{\sigmazd}{\sigma_T^0(z)}

\newcommand{\sigmazu}{\sigma_T^0(u)}

\newcommand{\epsilonzd}{\epsilon_T^0(z)}

\newcommand{\epsilonzu}{\epsilon_T^0(u)}
\newcommand{\sigmau}{\sigma(u)}
\newcommand{\epsilonu}{\epsilon(u)}

\newcommand{\rt}{h_T(x)}

\newcommand{\hnu}{h_\nu(x)}
\newcommand{\gnu}{g_\nu(x)}
\newcommand{\gt}{g_T(x)}
\newcommand{\ft}{f_T(x)}
\newcommand{\fnu}{f_\nu(x)}

\newcommand{\deltap}{z^\prime}

\newcommand{\deltapt}{z^\prime_T}

\newcommand{\deltapnu}{z^\prime_\nu}

\title{STABILITY OF AN ELASTODYNAMIC SYSTEM	WITH LOCALIZED INTERNAL DAMPING AND ACOUSTIC BOUNDARY CONDITIONS}

\newif\ifuniqueAffiliation
\uniqueAffiliationtrue

\ifuniqueAffiliation 
\author{ 
		\hspace{1mm}{Abdelkhalek Balehouane}\thanks{Corresponding author: Abdelkhalek Balehouane} \\
	Mathematics department,\\
	 USTHB, Bab-ezzouar,\\ 16111, Algiers, Algeria\\
	\texttt{balehouane.abdelkhalek@gmail.com} \\
	\And
	\hspace{1mm}{Hicham Kasri} \\
	Mathematics department,\\
	USTHB, Bab-ezzouar,\\ 16111, Algiers, Algeria\\
	\texttt{h.kasri@hotmail.com} \\
	\AND
\hspace{1mm}{Rokia Kechkar} \\
National School Of Artificial Intelligence\\
Sidi Abdellah, Algiers, Algeria\\
\texttt{kechkarr@gmail.com} \\
}
\else
\usepackage{authblk}

\author[1,2]{%
}
\fi

\renewcommand{\shorttitle}{\small{STABILITY OF AN ELASTODYNAMIC SYSTEM WITH ACOUSTIC BOUNDARY CONDITIONS}}

\begin{document}
	\maketitle

	\begin{abstract}
	In this paper, we prove a stability result for an elastodynamic system with acoustic boundary conditions and localized internal damping, defined in a bounded domain $\Omega$ of $\mathbb{R}^3$. Here, the internal damping is only assumed to be locally distributed and satisfies suitable assumptions. The smooth boundary of $\Omega$ is $\Gamma=\Gamma_0\cup\Gamma_1$ such that $\overline{\Gamma_0}\cap\overline{\Gamma_1}=\emptyset$. On $\Gamma_0$, we consider the homogeneous Dirichlet boundary condition, and on $\Gamma_1$ , we consider the acoustic boundary condition without a damping term. More precisely, by making use of semigroup techniques, well-posedness results are discussed, as well as the asymptotic behavior of solutions. The difficulty in establishing the stability of the system arises from the presence of higher-order operators, normal derivatives, and some boundary terms. The key tools combine the multiplier approach, trace theorems, ideas from Frota and Vicenté \cite{FrotaVicente2018}, and new technical arguments.
	\end{abstract}

	\keywords{Elastodynamic system \and Uniform decay rates \and Acoustic boundary conditions \and Localized internal damping}

\section{Introduction}
In the early acoustics literature, the Robin boundary condition was traditionally applied to the wave equation. Researchers soon recognized that this approach was not physically accurate for modeling realistic acoustic energy dissipation and absorption at boundaries. It often failed to capture the complex interplay between sound waves and the material properties of the interface. Addressing this limitation, Morse and Ingard \cite{MorseIngard} proposed the acoustic boundary conditions. These conditions were later rigorously developed by Beale and Rosencrans \cite{BaelRose74, Bael76}. \noindent Many authors have since examined the stability of systems that are subject to these boundary conditions (see, for example \cite{FrotaBraz,FrotaLarkin2005,FrotaMedeirosVicente2011, FrotaMedeirosVicente2014,FrotaVicente20152016,FrotaVicente2018, GalGoldsteinAcoustic2004,Graber, GraberSaidHouari2012, TaeGabHa2016,TaeGabHa2018,Vicente2016,VicenteFrotamemory2017}). The literature that is currently available takes a range of approaches: some deal with models that have negligible boundary mass or not, while others focus on the effect of boundary damping, which can be either linear or nonlinear, in the impenetrability condition. Further distinctions are made based on whether the damping acts only on a part of the boundary or internally throughout the whole domain. We provide a brief overview of a few representative contributions below. For further references, consult the bibliography and the works cited therein.

One direction explored in the literature modeling the acoustic boundary condition, was by assuming that each point on $\Gamma$ behaves as an independent spring, reacting to excess pressure without transverse tension between neighboring points of \( \Gamma \). This physical context corresponds to a locally reacting boundary. The first study conducted, in this context, was in \cite{BaelRose74}, where the authors considered the linear problem

\begin{equation}\label{system_Bael}
	\left\{
	\begin{aligned}
		u^{\prime\prime}  - \Delta u  &= 0 && \text{in } \Omega \times (0,\infty), \\
		\frac{\partial u}{\partial \nu} &= z^\prime && \text{on } \Gamma \times (0,\infty), \\
		f(x) z^{\prime\prime}  + g(x) z^\prime  +h z + \rho_0 u^\prime   &= 0 && \text{on } \Gamma \times (0,\infty), \\
	\end{aligned}
	\right.
\end{equation}
and proved that there is no uniform rate of decay to the energy associated. Since then, many authors have investigated problems involving acoustic boundary condition. Among these contributions, we emphasize the work in \cite{FrotaLarkin2005}, where using a degenerative approach, the authors established global solvability and decay estimates to a linear wave equation with an acoustic boundary condition in the case of negligible boundary structure.

\noindent In the context of porous acoustic boundary conditions, Graber, Said-Houari, and collaborators \cite{GraberSaidHouari2012} studied a semilinear wave equation involving nonlinear source terms and damping effects, both within the interior of the domain and on its boundary. Using nonlinear semigroup theory, they established the existence and uniqueness of a local-in-time solution, which, under suitable assumptions on the nonlinearities, can be extended to a global-in-time solution. Moreover, they investigated, in different cases, the decay rate and the blow up in finite time of the solution.

\noindent Frictional boundary damping on a subset of the boundary has also been studied massively in literature, notably in the work of \cite{Benyettou2016}, where authors considered a viscoelastic wave equation with source term under acoustic boundary conditions with viscoelastic damping term. \noindent They stated the polynomial decay and the blowup of solutions in different situations by using energy methods. In a subsequent contribution, considering a variable-coefficient wave, Considering variable-coefficient wave equation, and with nonlinear acoustic boundary conditions and source term, \cite{Jianghao2019} established energy decay results. \\

Motivated by a physical configuration in which the boundary $\Gamma$ is modeled as an elastic membrane, a broader framework has been proposed. This idea was first introduced in \cite{FrotaMedeirosVicente2011}, in which the authors considered a wave equation on the boundary governed by the Laplace–Beltrami operator. In that work, they established the existence, uniqueness, and exponential stability for a mixed system incorporating the Carrier equation
\begin{equation}\label{system_FrotaMeniros2011}
	\left\{
	\begin{aligned}
		u^{\prime\prime} - M\left( \int_{\Omega} u^2 \, dx \right) \Delta u + a u' + b |u'|^\alpha u' &= 0 && \quad \text{in } \Omega \times (0,1),\\
		u &= 0 && \text{on } \Gamma_0 \times (0, \infty),\\
		\frac{\partial u}{\partial \nu} &= \delta' && \text{on } \Gamma_1 \times (0, \infty),\\
		f(x) \delta'' - c^2 \Delta_{\Gamma} \delta + g(x) \delta' + h(x) \delta +\rho_0 u'&= 0 && \text{on } \Gamma_1 \times (0, \infty).
	\end{aligned}
	\right.
\end{equation}

\noindent Since then, multiple papers have addressed different formulation of problem \eqref{system_FrotaMeniros2011} under the non-locally reacting boundary condition. For instance, authors in \cite{VicenteFrotamemory2017} considered the global solvability to a problem involving the wave equation with memory term and acoustic boundary conditions. Taking into account  the boundary structure mass, they overcome the difficulty by mimicking the technique in \cite{Messaoudi2008} and established general energy decay under the effect of memory term in the impermability boudary condition. \noindent Another remarkable result was established in Ha~\cite{TaeGabHa2018}. Author considered a variable coefficients wave equation with nonlinear damping on the domain. Under suitable conditions on the nonlinear damping, they improved their previous results~\cite{TaeGabHa2016}, and obtained a general decay result. Moreover, we should make mention of \cite{JiangHao2020}, where using suitable Lyapunov functionals and energy compensation method, a general decay estimates was obtained to a variable-coefficient problem with viscoelastic damping and a memory term on the impermeability condition, which is an improvement of the results in \cite{TaeGabHa2018}.\\

In this work, we present a novel elastodynamic problem formulation with acoustic boundary conditions of non-locally reacting type. The conceptual foundation for this approach is based on the work of \cite{GalGoldsteinAcoustic2004}, who established, via parameter-dependent energy space isomorphisms, an equivalence between the general Wentzell boundary conditions and the acoustic boundary conditions in \cite{MorseIngard, BaelRose74, Bael76}. Such systems arise in the context of elastic wave propagation with dynamic interactions at the boundary. The interior domain models a linear elastic solid, while a portion of the boundary is subject to acoustic-type conditions that capture local reactive effects. This kind of setting occurs, for example, in seismic wave propagation, where the Earth's crust behaves elastically and part of the surface interacts with the atmosphere or absorbing layers through dynamic response. The interested reader is refereed to \cite{Sanchez1989, Sandberg2008, Nedelec, Peron} and references therein for more details.\\

\noindent Here, we are concerned with the following acoustic boundary value problem
\begin{equation}\label{syst_PDE}
	\begin{cases}u^{\prime \prime}-\operatorname{div} \sigma(u)+a(x) u^{\prime}=0, & \text { in } Q=\Omega \times] 0, \infty[, \\
		u=0, & \text { on } \left.\Sigma_0=\Gamma_0 \times\right] 0, \infty[, \\ 
		f_T(x)z_T^{\prime\prime}-\operatorname{div}_T\left(\sigma_T^0(z)\right)+g_T(x) z_T^{\prime}+r_T(x) z_T+u_T^{\prime}=0 & \text { on } \left.\Sigma=\Gamma_1 \times\right] 0, \infty[, \\
		f_\nu(x) z_\nu^{\prime \prime}+\sigma_T^0(z)\odot \partial_T \nu -\triangle_T z_\nu + s_\nu(x) z_\nu^{\prime}+r_\nu(x) z_\nu+u_\nu^{\prime}=0 & \text { on } \left.\Sigma=\Gamma_1 \times\right] 0, \infty[, \\
		\sigma(u) \cdot \nu=z^{\prime} & \text { on } \left.\Sigma=\Gamma_1 \times\right] 0, \infty[, \\ 
		u(., 0)=u_0, u^{\prime}(., 0)=u_1, & \text { in } \Omega,\end{cases}
\end{equation}
where $\Omega$ is a bounded domain of $\mathbb{R}^3$ with a boundary $\Gamma=\partial \Omega$ of class $C^2$, which is divided into two disjoint parts: $\Gamma=\Gamma_0 \cup \Gamma_1$. 

\noindent The main goal of this paper is to explore the asymptotic behavior of system (\ref{syst_PDE}). We begin by establishing the well-posedness of the problem using the Lumer-Philips theorem. Subsequently, we derive, under suitable assumptions on $\Gamma$ and the function $a(x)$, an exponential stability result through the multiplier method. This method proves to be highly effective in solving such problems; it is based on energy estimates and Gronwall inequalities. Let us mention some seminal works utilizing this method. In \cite{Lionstome1}, J.-L. Lions developed the multiplier method as a powerful tool for studying the exact controllability of hyperbolic and parabolic problems. He derived observability estimates, ensuring that the total energy of the system could be controlled from a subset of the domain or its boundary. The pioneering work of Bardos et al. \cite{BLR} provides sufficient geometric conditions on the control region for exact controllability and stabilization to hold. Specifically, they require that every geometric optic ray intersects the control region. Later on, Burq and Gérard \cite{BG} demonstrated that these results remain valid under weaker regularity assumptions on the domain and operator coefficients. While these geodesic conditions are generally not explicit, they enable the derivation of energy decay estimates under some hypotheses. In \cite{Zuaz}, Zuazua establishes a geometric controllability condition for the semilinear wave equation with localized damping: $\partial_{tt}u-\Delta u+a(x)\partial_t u+f(u)=0$ in $\Omega\times\mathbb{R}_+$, where $a(x)\geq 0$ is a damping coefficient supported on $\omega\subset\Omega$. The principal result requires that: $\{x\in\Omega|\ dist(x,\partial\Omega\leq\varepsilon)\}\subseteq\omega$ for some $\varepsilon>0$, as we have considered in this work. Subsequent work by Fu et al. \cite{FYZ}, combined with Zuazua's framework \cite{Zuaz}, extends the exponential stabilization to geometries where: $\{x\in\Omega|\ (x-x_0)\cdot\nu\leq0\}\subseteq\omega$, with $x_0\in\mathbb{R}^n$  being an observation point and $\nu$ the outward normal. This geometric condition explicitly excludes damping configurations vanishing at antipodal points of spherical domains. The piecewise multiplier method developed by Liu \cite{Kliu} provides a generalization of these results by admitting multiple observation points $x_{0i}\in\mathbb{R}^n$, $i=1,\cdots, I$ (distinct observation points), and relaxing the geometric constraints to finite unions $\omega$ 
$$
\omega \supset \mathscr{N}_{\varepsilon}\left(\cup_{i=1}^I \Gamma_i\left(x^i\right) \cup\left(\Omega \backslash \cup_{i=1}^I \Omega_i\right)\right),
$$
where $
\Gamma_i\left(x^i\right)=\left\{x \in \partial \Omega_i,\left(x-x_0^i\right) \cdot v_i(x) \geq 0\right\}
$.

Following from our previous analysis, the question we address in this paper is the following: Can the system be stabilized by introducing a localized feedback acting only on a portion of the surface? If so, what geometric conditions must this surface satisfy? 

The principal contribution of this paper consists in establishing affirmative answers to both questions under appropriate geometric conditions. In fact, the difficulties of our work are summarized as follows:
\textbf{(i)} In the analysis of decay properties for systems with acoustic boundary conditions on non-locally reacting boundaries, one difficulty is to manipulate the term involving $\sigma(u):\varepsilon(u)$ on $\Gamma$ which incorporates both the normal derivative component $(\sigma(u)\cdot\nu)$ and the tangential stress-strain $\sigma_T(u):\varepsilon_T(u)$ (see Appendix); \textbf{(ii)} The acoustic boundary conditions for non-locally reacting boundaries present analytical challenges comparable to Wentzell boundary conditions (see for instance, \cite{GalGoldsteinAcoustic2004,Cav2009,Cav2012,KasriHemina2016,KasriHemina2017}) while simultaneously elevating the energy framework, as the natural energy functional incorporates the $H^1(\Gamma)$ Sobolev norm of the boundary trace (see (\ref{egalit_energy}) below). This energy enhancement necessitates examining whether a damping mechanism acting away from a portion of the boundary $\Gamma$, must effectively control the higher-order potential energy component despite acting at a distance; \textbf{(iii)} The impenetrability condition in equation $(\ref{syst_PDE})_{5}$ lacks a damping term, in contrast to several well-established works in the field (see, e.g., \cite{FrotaBraz,Graber, LP}). This absence introduces a significant challenge in the context of problem stabilization (the undamped portions of the boundary $\Gamma$ are subject to acoustic  boundary conditions); \textbf{(iv)} The inclusion of the term $z''$ in the dynamic Equ $(\ref{syst_PDE})_{3,4}$ (non-negligible mass i.e., $f\geq0$) complicates the derivation of the energy estimate. Indeed, at a key stage in the proof, after applying suitable multipliers, we arrive at an integral term of the form $\int_{\Gamma_1}f(x)\div_T(x-x_0)_T|z'|^2\ dS dt+\int_{\Gamma_1}f(x)z'(z_\nu\partial_T\nu)(x-x_0)_T\ dSdt$, which introduces analytical challenges; \textbf{(v)} The presence of the terms $u'_T$ and $u'_\nu$ in Equ $(\ref{syst_PDE})_{3,4}$ thus poses difficulties for stability analysis, because $u'$ in this equation is the trace of $u'$ defined on $\Omega$. A precise estimation of this term requires additional regularity, which is not guaranteed by the existence theorem. Our strategy is to combine the results due to Lions \cite{Lionstome1} with ideas from the work by Frota and Vicenté \cite{FrotaVicente20152016} and, mainly, those of Cavalcanti et al. \cite{Cav2009}, adapted to our situation, which are the key to proving our main result.

The plan for the rest of the paper is as follows: Section 2 introduces the necessary notations and assumptions needed in our work.  In Section 3, we state the problem and present our main results, including the well-posedness and stability theorems. Section 4 is devoted to the proof of the existence and uniqueness of solutions using semigroup theory. Section 5 establishes the uniform stability results using the multiplier methods.
	
	\section{Notations and assumptions}
	In this section, we introduce some notations and assumptions that will be used later. We denote by ' the time derivative, $z$ models the normal displacement of the point $x\in\Gamma$ in the time $t$ and $f, g, h$ are given real valued bounded functions defined on $\Gamma_1$, $\div_T$ stand for the tangential divergence, $\nu$ is the outward normal unit vector on $\Gamma_1$, $\sigma(u)=\left(\sigma_{i j}(u)\right)_{i, j=1}^3$ is the stress tensor given by $\sigma(u)=2 \alpha \varepsilon(u)+\lambda \operatorname{div}(u) I_3$, where $\lambda, \alpha>0$ are the Lam\'e coefficients, $I_3$ is the identity matrix of $\mathbb{R}^3$ and $\varepsilon(u)=\frac{1}{2}\left(\nabla u+(\nabla u)^T\right)=\left[\varepsilon_{i j}(u)\right]_{i, j=1}^3$ is a $3 \times 3$ symmetric matrix. From now on, a summation convention with respect to repeated indexes will be used.
	
	$\Delta$, $\div$ and $\nabla$ are the Laplacian operator, divergence, and gradient for spatial variables, respectively. On the other hand, $\Delta_{T}$, $\div_{T}$ and $\nabla_{\Gamma}$ represent the Laplace-Beltrami operator, the tangential gradient and the tangential divergence on $\Gamma_1$, respectively such that the following Stokes formula holds 
	\begin{equation}\label{stokes_formula}
		\int_{\Gamma_1}\div_T v_T\cdot u_T\ dS=-\int_{\Gamma_1} tr(v_T\cdot\pi \partial_{T} u_T)\ dS,\ \forall u_T, v_T \in H^1_0(\Gamma_1)^3,
	\end{equation}
	where $"tr"$ means the trace of a matrix and $\pi(x)$ is the orthogonal projection on tangent plane $T_x(\Gamma)$. Moreover, let $x \in \Gamma$, we denote by $\pi(x)$ the projection from $\mathbb{R}^3$ on the tangent plane $T_x(\Gamma)$ and for a given vector field $\varphi$ we have, for any point $x$ of $\Gamma, \varphi(x)=\varphi_T(x)+$ $\varphi_\nu(x) \nu(x)$, where $\varphi_T(x)=\pi(x)(\varphi(x))$ is the tangential component of $\varphi, \varphi_\nu(x)=$ $\varphi(x) \cdot \nu(x)$ and $\nu$ represents the unit outward normal field to $\Gamma, \sigma_S(\varphi)=2 \alpha \varepsilon_S(\varphi)$, $\sigma_\nu(\varphi)=(\lambda+2 \alpha) \varepsilon_\nu\left(\varphi_\nu\right)+\lambda \operatorname{tr}_2\left(\varepsilon_T(\varphi)\right)$, and $\sigma_T^0(\varphi)\odot \partial_T \nu=\operatorname{tr}\left(\sigma_T^0(\varphi) \cdot \partial_T \nu\right)$, with

	\begin{equation}\label{defini_sigma0_T}
		\sigma_T^0(\varphi)=2 \mu \varepsilon_T^0(\varphi)+\lambda^* \operatorname{tr}\left(\varepsilon_T^0(\varphi)\right) I_2,\quad \lambda^*=2\lambda \mu(\lambda+2 \mu)^{-1},
	\end{equation}

	\noindent where 
	\begin{equation}\label{defini_varepsilon0_T}
		\varepsilon_T^0(\varphi)=\varepsilon_T(\varphi)=\frac{1}{2}\left(\pi \partial_T \varphi_T \pi+\left(\pi \partial_T \varphi_T \pi\right)^t\right)+\varphi_\nu\left(\partial_T \nu\right),
	\end{equation}
	
	\begin{equation}\label{defini_varepsilon_S}
		\varepsilon_S(\varphi)=\frac{1}{2}\left(\partial_\nu \varphi_T-\left(\partial_T \nu\right) \varphi_T+(\partial_T \varphi_\nu)^t)\ \text{ and }\ \varepsilon_\nu(\varphi)=\partial_\nu \varphi_\nu\right. ,
	\end{equation}
	
	\noindent  in which $\partial_T \nu$ denote the curvature operator on $\Gamma, \partial_T$ (respectively $\partial_\nu$ ) denote the tangential (respectively normal) derivative, $I_2$ is the identity of the tangent plane and "tr" means the trace of a matrix. Furthermore, $\operatorname{div}_T \sigma_T^0(u)$ is the tangential divergence of the endomorphism field $\sigma_T^0(u)$, and $a(x)\in L^{\infty}(\Omega)$ is a nonnegative coefficient of the damping term. From now on, we will always assume the following conditions:
	
	\noindent \textbf{Assumption 1.} The function $a(x)$ satisfies
	\begin{equation}
		a(x)\in L^{\infty}(\Omega),\ \ a(x) \geq a_0 > 0 \text { over } \omega \subset \Omega\tag{H1}, \label{cond_def_a_damping}
	\end{equation}
	where $\omega$ is a neighborhood of the boundary $\Gamma_0$ and $a_0$ is a positive constant.

	\noindent \textbf{Assumption 2.} The functions $f(x)$, $g(x)$, $h(x):$ $\Gamma_1\rightarrow\mathbb{R}^+$ are essentially bounded such that
	\begin{equation}
		f(x)\ge f_0, g(x)\ge g_0, h(x)\ge h_0 \tag{H2} \label{cond_f_g_h_ge_f0_g0_h0},
	\end{equation}
	where $f_0$, $g_0$ and $h_0$ are positive constants. Moreover, we assume that the following geometrical conditions hold true (see Figure \ref{drawing1}). 

	\noindent \textbf{Assumption 3.}
	There exists some $x_0\in\mathbb{R}^3$ and $\delta>0$ such that 
	\begin{align}
		\Gamma_0&=\{ x\in\Gamma;\text{  } (x-x_0)\cdot\nu\geq \delta>0\}\tag{H3},\label{geom_cond1}\\
		\Gamma_1&=\{x\in\Gamma;\text{ }(x-x_0)\cdot\nu\leq0\}\tag{H4}.\label{geom_cond2}
	\end{align}
	\begin{figure}[!h]
		\begin{center}
			\includegraphics[scale=0.4]{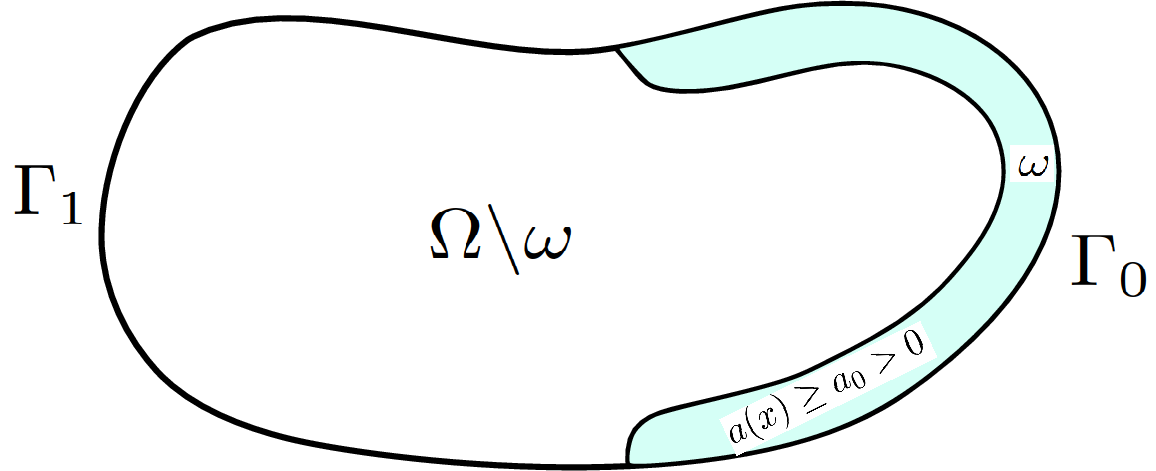}
			\captionof{figure}{Example of a domain $\Omega$ in the $\mathbb{R}^2$ case and localized damping region}
			\label{drawing1}
		\end{center}
	\end{figure}
\section{Statement of the problem and Main results}
To establish the existence and uniqueness of solutions to the main problem~\eqref{syst_PDE}, we formulate the problem within the framework of linear semigroup theory. Particularly, we reformulate system \eqref{syst_PDE} as an abstract Cauchy problem on an appropriately chosen Hilbert space. 
\noindent We define the phase space associated with our system (\ref{syst_PDE}) by
\begin{equation*}
	\mathbb{H}=\mathbb{V}\times L^2(\Omega)^3\times H^1_{\Gamma_1}\times L^2(\Gamma_1)^3,
\end{equation*}
where $\mathbb{V}$ is a Hilbert space defined by
\begin{equation*}
	\mathbb{V}=\left\{u\in H^1(\Omega)^3;\>\>u_{\mid_{\Gamma_0}}=0\right\},
\end{equation*}
equipped with the norm inducted by inner product in $\mathbb{V}$ as
\begin{equation}\label{norm_V}
	\langle u_1,u_2\rangle_{\mathbb{V}}=\int_\Omega\sigma(u_1)\odot\varepsilon(u_2)\ dx,
\end{equation}

\begin{equation}
	\langle z_1, z_2\rangle_{H^1_{\Gamma_1}}=\intgg{h(x)z_1\ z_2} + \intgg{\sigma_T^0(z_1)\odot\epsilon_T^0(z_2)} + \intgg{\nabla_T z_{1\nu}.\nabla_T z_{2\nu}},
\end{equation}
and 
\begin{equation}
	\langle w_1, w_2\rangle_{L^2_{\Gamma_1}}=\intgg{f(x)w_1.w_2} ,
\end{equation}
for which $``\odot"$ is defined as $\displaystyle\mathcal{A}\odot\mathcal{B}=\sum_{i,j=1}^{3}a_{ij}b_{ij}$ for real matrices $\mathcal{A}=(a_{ij})$ and $\mathcal{B}=(b_{ij})$.
\noindent Let us introduce a Poincaré-type inequality on the boundary $\Gamma_1$. This inequality will serve  in the upcoming sections.
\begin{remark}
	\item Let $\mathcal{C}_p$ be the smallest positive constant such that 
	\begin{equation}\label{poincar}
		\intgg{|u|^2\ dS}\leq \mathcal{C}_p^2\|u\|^2_{\mathbb{V}},\quad\forall u\in\mathbb{V}.
	\end{equation}
\end{remark}
\noindent Next we introduce the vector function $U=\left(u,v,z,w_T,w_\nu\right)^t$, then the equations in $\eqref{syst_PDE}$ is associated with the following abstract problem
\begin{equation}\label{prr1}
	\left\{
	\begin{array}{ll}
		\dot{U}(t)+\mathbb{A}U(t)=0, & \hbox{   } \\
		U(0)=U_0=(u_0,v_0,z_0, w_{0T}, w_{0\nu}), & \hbox{   }
	\end{array}
	\right.
\end{equation}
in which
\begin{equation}
	\forall U\in\mathbb{H},\ \displaystyle\mathbb{A}U=\left(
	\begin{array}{c}
		-v \\[1ex]
		-\div\sigma(u)+a(x)v \\[1ex]
		-w\\[1ex]
		f_T^{-1}\left(-\div_T(\sigmazd)+ g_T(x)w_T+ h_T(x)z_T+v_T\right)\\[1ex]
		f_\nu^{-1}\left(\sigmazd\odot\partial_T\nu - \triangle_Tz_\nu+g_\nu(x)w_\nu+h_\nu(x)z_\nu+v_\nu\right)
	\end{array}
	\right),
\end{equation}
where we have $u^\prime=v$ and $z^\prime=w$. The domain of $\mathbb{A}$ is given by
\begin{equation*}
	\mathcal{D}(\mathbb{A})=\left\lbrace
	\begin{array}{l}
		(u,v,z, w_T, w_\nu)\in\mathbb{H},\  v\in\mathbb{V},w\in L^2(\Gamma_1),\ \div\sigma(u)\in L^2(\Omega)^3,\sigmau\cdot\nu=w \text{ on } \Gamma_1\\[1ex]
	\end{array}
	\right\rbrace .
\end{equation*}
Thus, solving the system $\eqref{syst_PDE}$ is equivalent to solve the abstract Cauchy problem $(\ref{prr1})$ associated to the operator $\mathbb{A}$. Our solvability result is the following
\begin{theorem}\label{theorem_existence_uniqueness}
	Assume that assumptions \eqref{cond_f_g_h_ge_f0_g0_h0} and \eqref{cond_def_a_damping} hold, and suppose that $(u_0,v_0,z_0,w_{0T},w_{0\nu})\in \mathcal{D}(\mathbb{A})$, then the problem $\eqref{syst_PDE}$ has a unique (strong) solution in the class
	\begin{equation*}
		(u,v,z,w_{T},w_{\nu})\in \mathcal{C}(\mathbb{R}_+,\mathcal{H}).
	\end{equation*}
	Moreover, for all $(u_0,u_1,z_0,z_{0T},z_{0\nu})\in \mathbb{H}$, the problem  $\eqref{syst_PDE}$ has a unique (weak) solution in the class
	\begin{equation*}
		(u,v,z,w_T,w_\nu)\in W^{1,\infty}(\mathbb{R}_+,\mathcal{H})\cap L^\infty(\mathbb{R}_+, \mathcal{D}(\mathbb{A})).
	\end{equation*}
\end{theorem}
On the other hand, our  main stability results is stated in the following theorem
\begin{theorem}\label{theorem_uniform_stability}
	Assume that assumptions \eqref{cond_def_a_damping}, \eqref{cond_f_g_h_ge_f0_g0_h0}, \eqref{geom_cond1} and \eqref{geom_cond2} hold. Then for each set of initial conditions $(u_0, z_0)\in \mathbb{V}\times L^2(\Omega)^3$ there exist two positive constants $ \mathcal{K}_1$ and $\mathcal{K}_2$ such that 
	\begin{equation}
		E(t)\le \mathcal{K}_1 e^{-\frac{t}{\mathcal{K}_2}}E(0),  \forall t\ge 0, 
	\end{equation} 
	for all $(u, z)$ solutions of \eqref{syst_PDE}. 
\end{theorem}
\section{Existence and uniqueness results: Proof of Theorem\ref{theorem_existence_uniqueness}}
In this part of the paper we prove the well-posedness to problem~\ref{syst_PDE}. To achieve this goal, we rely on the semi group theory (see \cite{Pazylivre}). We need to establish that operator $\mathbb{A}$ is maximal monotone. We prove that A is monotone through Lemma~\ref{Lemma_monotonie} and that $(I-\mathbb{A})$ is surjective via Lemma~\ref{Lemma_maximal}.
\begin{lemma}\label{Lemma_monotonie}
	For all $U\in \mathcal{D}(\mathbb{A})$, we have
	\begin{equation*}
		\langle\mathbb{A}U,U\rangle_{\mathbb{H}}\geq0.
	\end{equation*}
\end{lemma}
\begin{proof}
	Let be $U\in \mathcal{D}(\mathbb{A})$, by the definition of the inner product in $\mathbb{V}$ and $L^2(\Gamma_1)$, the boundary conditions in $\eqref{syst_PDE}$, Green's and Stokes formulas \eqref{stokes_formula}, one obtains
	\begin{align*}
		\langle\mathbb{A}U,U\rangle_{\mathbb{H}}&= -\intgo{\sigmau\odot\epsilon(v)} - \intgo{\div\sigma(u)\cdot v} +\intgo{a(x)|v|^2} - \intgg{h(x)w\cdot z}\\ &-\intgg{\sigma_T^0(w)\odot\epsilon_T^0(z)} +\intgg{\triangle_T z_\nu\cdot z} - \intgg{\div_T(\sigmazd)w_T} +\intgg{g_T(x)|w_T|^2}\\ &+\intgg{r_T(x)z_T\ w_T} +\intgg{v_Tw_T}+\intgg{\sigmazd\odot\partial_T\nu\cdot _\nu}+ \intgg{\gnu|w_\nu|^2}+ \intgg{\hnu z_\nu\cdot w_\nu},\\
		&+\intgg{u^\prime_\nu w_\nu}-\intgg{\triangle_Tz_\nu\ w_\nu}\\
		&=\intgo{a(x)|v|^2} +\intgg{g(x)|w|^2} \ge 0.	
	\end{align*}
	Thus, 
	\begin{align*}
		\langle\mathbb{A}U,U\rangle_{\mathbb{H}}&= -\intgo{\sigmau\odot\epsilon(v)} -\intgg{(\sigmau\cdot\nu)\cdot v} +\intgo{\sigma(u)\odot\epsilon(v)}  
		+\intgo{a(x)|v|^2} - \intgg{h(x)w\cdot z}\\ 
		&-\intgg{\sigma_T^0(z)\odot\epsilon_T^0(w)} -\intgg{\nabla_T z_\nu\cdot \nabla_Tw_\nu}
		+ \intgg{\sigmazd\odot\epsilon_T^0(w_T)} +\intgg{g_T(x)|w_T|^2}\\
		&+\intgg{h_T(x)z_T\ w_T} +\intgg{v_Tw_T}+\intgg{(\sigmazd\odot\partial_T\nu)\cdot w_\nu}+ \intgg{\gnu|w_\nu|^2}\\
		&+ \intgg{\hnu z_\nu\cdot w_\nu} +\intgg{v_\nu w_\nu} +\intgg{\nabla_Tz_\nu\cdot \nabla_Tw_\nu},\\
		&=\intgo{a(x)|v|^2} +\intgg{g(x)|w|^2} \ge 0,	
	\end{align*}
	\noindent then $\mathbb{A}$ is monotone. 
\end{proof}
\begin{lemma}\label{Lemma_maximal}
	The operator \( \mathbb{A} \colon D(\mathbb{A}) \subset \mathbb{H} \to \mathbb{H} \) satisfies the range condition
	\[
	\operatorname{Range}(I + \mathbb{A}) = \mathbb{H}.
	\]
	In particular, \( \mathbb{A} \) is maximal monotone and hence generates a \( C_0 \)-semigroup on \( \mathbb{H} \).
	\begin{proof}
		The proof consists in verifying the surjectivity of the operator \( I + \mathbb{A} \) by explicitly solving the resolvent equation \( (I + \mathbb{A})U = k \) for all \( k \in \mathbb{H} \). This is achieved by reducing the problem to a variational formulation and applying the Lax–Milgram lemma, which guarantees the existence and uniqueness of the solution. Clearly, $D(A)$ is dense in $\mathbb{H}$. \noindent Let $k:=(k_1,k_2,k_3,k_4,k_5)^t\in\mathbb{H}$, we have to solve
		\begin{equation}\label{egalit_surje}
			U:=(u,v,z,w_T,w_\nu)\in\mathcal{D(\mathbb{A})},\quad U+\mathbb{A}U=k,
		\end{equation}
		which means that 
		\begin{align}
			&u-v=k_1,\label{e1}\\
			\vspace{0.1cm}
			&v-\div\sigma(u)+a(x)v=k_2,\label{e2} \\
			\vspace{0.1cm}
			&z -w = k_3,\label{e3} \\
			\vspace{0.1cm}
			&w_T+ \frac{1}{f_T}\left(-\div_T(\sigmazd)+ g_T(x)w_T+ r_T(x)z_T+u^\prime_T\right) = k_4, \label{e4}\\
			&w_\nu + \frac{1}{f_\nu}\left(\sigmazd\odot\partial_T\nu - \triangle_Tz_\nu+s_\nu(x)w_\nu+r_\nu(x)z_\nu+u^\prime_\nu\right)=k_5 . \label{e5}
		\end{align}
		It can be easily shown that $v$ and $w$ satisfy 
		\begin{align}\label{variat_form_surjec}
			\left\{
			\begin{array}{ll}
				v-\div\sigma(v)+a(x)v= \xi_1,\\
				f_Tw_T -\div_T\sigma_T^0(w)+\gt w_T+\rt w_T +v_T = \xi_2, \\
				f_\nu w_\nu+\sigma_T^0(w)\odot\partial_T\nu -\triangle_T w_\nu + \gnu w_\nu+ \hnu w_\nu +v_\nu= \xi_3,\\
			\end{array}
			\right.
		\end{align}
		where $$\xi_1= k_2+\div\sigma(k_1),$$ $$\xi_2= f_Tk_4+\div_T\sigma_T^0(k_3) -\rt k_{3T},$$  $$\xi_3= f_\nu k_5-\sigma_T^0(k_{3\nu})\odot\partial_T\nu+\triangle_T k_{3\nu}-\hnu k_{3\nu}.$$
		We consider the variational formulation corresponding to $\eqref{variat_form_surjec}$ given by 
		\begin{equation}\label{variat_form_surjec1}
			{\Phi}\left((v,w_T, w_\nu), (\tilde{v},\tilde{w_T}, \tilde{w_\nu})\right)= \Psi(\tilde{v},\tilde{w_T}, \tilde{w_\nu}),
		\end{equation}
		with 
		${\Phi}:\mathbb{V}\times \mathbb{H}^1_{\Gamma_1}$, the bilinear form defined as
		\begin{align*}
			\Phi\left((v,w_T, w_\nu), (\tilde{v},\tilde{w_T}, \tilde{w_\nu})\right)&=\intgo{v\tilde{v}} + \intgo{\sigma(v)\odot\epsilon(\tilde{v})}  +\intgg{a(x)v\tilde{v}}\\
			& +\intgg{\left(f(x)+g(x)+h(x)\right)w\tilde{w}} +\intgg{\nabla_T w_\nu\cdot \nabla_T \tilde{w}_\nu}\\
			& +\intgg{\sigma_T^0(w)\odot\epsilon_T^0(\tilde{w})} 
		\end{align*}
		and the linear form $\Psi$ such that
		\begin{equation*}
			\Psi((v,w), (\tilde{v},\tilde{w}))= -\intgo{\sigma(k_1)\odot\epsilon(\tilde{v})}+\intgg{k_2\tilde{v}} -\intgg{\sigma_T^0(k_3)\odot\epsilon_T^0(\tilde{w})} -\intgg{h(x)k_3\tilde{w}} .
		\end{equation*}
		The boundedness of $\Phi$ is straightforward. Furthermore, we can readily demonstrate that $\Phi$ is coercive. Thus, from $\eqref{cond_f_g_h_ge_f0_g0_h0}$ we get 
		\begin{align}
			\Phi\left((v,w_T,w_\nu),(v,w_T,w_\nu)\right)&\ge \intgo{v^2}+ \intgo{\sigma(v)\odot\epsilon(v)} + \intgg{(f_0+g_0)(x)w^2}\nonumber\\
			& + h^{-1}_0\intgg{h(x)w^2}+ \intgg{(\nabla_Tw_\nu)^2} + \intgg{\sigma_T^0(w)\odot\epsilon_T^0(w)},\nonumber\\
			&\ge C\left[\|v\|^2_{\mathbb{V}} + \|w\|^2_{H^1_{\Gamma_1}}\right],
		\end{align}
		where $C=min\left\{1,h^{-1}_0\right\}$.	
	
		\noindent Moreover, $\Psi$ is continuous form on $\mathbb{V}\times H^1_{\Gamma_1}$. Then, it follows by Lax-Milgram's theorem that $\eqref{variat_form_surjec1}$ admits a unique solution $(v,w)\in \mathbb{V}\times H^1_{\Gamma_1}$ for every $k\in\mathbb{H}$. Next we reconstruct the remaining part of the vector $U$. From equations \eqref{e1}-\eqref{e5}, we have $u\in \mathbb{V}$, $\div\sigma(u)\in L^2(\Omega)$ and $z\in H^1_{\Gamma_1}$. 
		
		\noindent In conclusion, for all $k\in\mathbb{H}$ we obtain $U =(u,v,z,w_T,w_{\nu})\in D(\mathbb{A})$ such that $(I+\mathbb{A})U=k$. Hence, we claim by the Lumer-Phillips theorem (cf. \cite{Pazylivre}) that $\mathbb{A}$ generates a C$_0$-
		semigroup of contractions $(e^{t\mathbb{A}})_{t\geq 0}$ in $\mathbb{H}$.
	\end{proof}
\end{lemma}
\section{Uniform stability results:}
The aim of this part is to prove theorem~~\ref{theorem_uniform_stability}. This decay result is established using the multiplier method \cite{Lionstome1} and follows the analytical framework developed in \cite{Komornik1994}. In order to structure the ideas, this section is divided into three main parts, in the first part we define the energy associated to our problem and provide useful lemmas to the subsequent analysis. The second part focuses on deriving the fundamental identity. The last part consists on establishing a series of estimates leading to our decay result. 

\subsection{Energy equality}
Let us introduce the energy functions representing both the internal and boundary energy of the system \eqref{syst_PDE} as
\begin{equation}\label{egalit_energy}
	E(t)= \frac{1}{2} \left[\underbrace{\intgo{\left(|u'|^2+\sigma(u)\odot\epsilon(u)\right)}}_{E_\Omega(t)}+ \underbrace{\intgg{\left(|f(x)||z'|^2  + h(x)|z|^2 +\sigmazd\odot\epsilonzd +|\nabla_T z_\nu|^2\right)}}_{E_\Gamma(t)} \right].
\end{equation}
\begin{lemma}\label{lemma_energ_decrea}
	The energy defined in $\eqref{egalit_energy}$ is non increasing and satisfies $\forall s_1, s_2 $ such that $s_2> s_1\ge 0$, we have 
	\begin{equation}\label{egalite_energy_decreas}
		E(s_2)- E(s_1)=-\left(\ints{\intgo{a(x)|u'|^2}} +\ints{\intgg{g(x)|z'|^2}}\right) .
	\end{equation}
\end{lemma}
\begin{proof}
	Multiplying the first equation of $\eqref{syst_PDE}$ by $u'$, integrating over $\Omega$ by parts and using boundary conditions in $\eqref{syst_PDE}$, gives us
	\begin{align*}
		\frac{1}{2}\frac{d}{dt}&\left[\intgo{\left(|u'|^2 + \sigmau\odot\epsilonu\right)}+ \intgg{\left(f(x)|z'|^2 + h(x)|z|^2 + \sigmazd\odot\epsilonzd+|\nabla_T z_\nu|^2\right)}  \right]\\ 
		& =-\intgo{a(x)|u'|^2} -\intgg{g(x)|z'|^2},
	\end{align*}
	which can be rewritten as 
	\begin{equation*}
		\frac{d}{dt} E(t)= -\intgo{a(x)|u'|^2} -\intgg{g(x)|z'|^2}.
	\end{equation*}
\end{proof}

\begin{lemma}\label{Lemma_estima_q_nabla_u_u}
	Let $(u, z)$ be a strong solution of \eqref{syst_PDE}, then for all $t\ge 0$, there exists a positive constant $\xi$ such that 
	\begin{equation}\label{estima_q_nabla_u_u}
		\left\|2\Big[(x-x_0):\nabla u+u\Big]^2\right\|_{(L^2(\Omega))^3}\le \xi E(t).
	\end{equation}  
\end{lemma}
\begin{proof}
	Putting for brevity: $R=\|x-x_0\|_{L^\infty(\Omega)}$. 
	By a straightforward calculation we get
	\begin{equation*}
		\begin{split}
			2\intgo{ \Big((x-x_0):\nabla u+u\Big)^2}=& \intgo{ \Big[|2(x-x_0)\cdot\nabla u|^2+ 4|u|^2+ 8u\cdot (x-x_0)\cdot\nabla u\Big]},\\
			=& \intgo{\Big[|2(x-x_0):\nabla u|^2+ 4|u|^2+ 4 (x-x_0)\cdot\nabla (|u|^2)\Big]}.
		\end{split}
	\end{equation*}
	Using Green's formula, we obtain
	\begin{equation*}
		\begin{split}
			2\intgo{ ((x-x_0):\nabla u+u)^2}&=  \intgo{\Big[|2(x-x_0):\nabla u|^2- 8|u|^2\Big]}+ 4 \intgg{(x-x_0)\cdot\nu |u|^2},\\
			&\le 4R^2 \intgo{|\nabla u|^2}+ 4R \intgg{ |u|^2}.
		\end{split}
	\end{equation*}
	It follows from classical Korn's inequality that 
	\begin{equation*}
		2 \intgo{\Big((x-x_0):\nabla u+u\Big)^2} \le 4R^2 C\intgo{\sigmau\odot\epsilonu} + 4R\intgo{|u|^2}.
	\end{equation*}
	Using \eqref{poincar}, we infer
	\begin{align*}
		2\int_{\Omega} &\left((x-x_0):\nabla u+u\right)^2 dx\le \left(4R^2C +\mathcal{C}_p^2\right)\int_{\Omega}\sigmau\odot\epsilonu dx\\
		&+\mathcal{C}_p^2\left[\intgg{ \sigmazu\odot\epsilonzu} + \intgg{ |\nabla_T u_\nu|^2}\right]\le \xi E(t),
	\end{align*}
	and so \eqref{estima_q_nabla_u_u} holds with $\xi=C(R, \mathcal{C}_p)$.
\end{proof}
\noindent In order to handle tangential regularity, the next lemma provides a norm equivalence, which will be used later. For the detailed proof, the interested reader may see Proposition 2.1 \cite{Hemina1999}.
\begin{lemma}\label{Lemme_equiva_norm}
	We define in  $H^1(\Gamma_1)$ a norm given by 
	\begin{equation*}
		\| z_T\|^2_{H^1(\Gamma_1)}= \intgg{\Big(| z_T|^2+\epsilonzd\odot\epsilonzd\Big)},
	\end{equation*}
	which is equivalent to the norm $$z_T\rightarrow \left(\|z^1\|^2_{L^2(\Gamma_1)}+\|z^2\|^2_{L^2(\Gamma_1)}\right)^{\frac{1}{2}},$$
	where $z_T=z^1a_1+z^2a_2$,  $\| z_T\|^2_{L^2(\Gamma_1)}= \intgg{| z_T|^2}$ and $\epsilonzd=\frac{1}{2}\left(\pi\partial_T  z_T\pi+ (\pi\partial_T  z_T\pi)^t \right).$
\end{lemma}
\begin{lemma}\label{Lemma_estima_nabla_T_delta_nu}
	Let $(u, z)$ be a strong solution of \eqref{syst_PDE}. Then for all $\tau>0$ (to be chosen later) there exists a positive  constant $\widehat{C}=\widehat{C}(\tau)$ independent on $T$ such that
	\begin{equation}\label{estima_nabla_T_delta_nu}
		\intt{\intgg{\left[ \sigmazd\odot\epsilonzd+| z|^2+ |\nabla_T z_\nu|^2\right]}}\le \tau^{-1}\widehat{C} E(0) +\tau\intt{E(t)}.
	\end{equation}
\end{lemma}
\begin{proof}
The proof follows the strategy in Lemma~8.12, \cite{Komornik1994}. The idea relies on the construction of an elliptic problem and to use suitably designed multiplier associated to this auxiliary problem. Through some computations, we obtain estimations of the desired terms. This approach is widely used in the literature, the reader is referred to the works  \cite{Komornik1994,Martinez1999} for further details.
\end{proof}
\subsection{The multiplier identity}
For simplicity, we will work with regular solutions of $\eqref{syst_PDE}$, and through standard density arguments, we can also extend our results to weak solutions.
\begin{lemma}\label{lemme_identit_fondam_final_1}
	For every strong solution $(u, z)$ of $\eqref{syst_PDE}$ we have the following identity
	\begin{align}\label{egalite_fond_GLOB}
		&2\intt{E(t)}       =     -2\left[\intgo{u^\prime\left((x-x_0):\nabla u+ u\right)}\right]_0^T -2\intt{\intgo{a(x)u^\prime\left((x-x_0)\cdot\nabla u+u\right)}}\nonumber\\
		&{+2\intt{\intgg{ z^\prime\left((x-x_0):\nabla u+u\right)}}} +{\intt{\intgg{((x-x_0)\cdot\nu)\left(|u^\prime|^2-\sigmau\odot\epsilonu\right)}}} \nonumber\\
		&-2\left[\intgg{\ft\deltapt(\pi\partial_T z_T\pi)(x-x_0)_T}\right]_0^T-2\left[\intgg{\ft\deltapt( z_\nu\partial_T\nu)(x-x_0)_T}\right]_0^T\nonumber\\
		&-2\left[\intgg{f_\nu(x)\deltapnu\partial_T z_\nu\ (x-x_0)_T}\right]_0^T +2\left[\intgg{f_\nu(x)\deltapnu(z^t_T\partial_T\nu)\ (x-x_0)_T}\right]_0^T\nonumber\\
		& -\intt{\intgg{\ft\div_T(x-x_0)_T |\deltapt|^2}} +2\intt{\intgg{\ft\deltapt(\deltapnu\partial_T\nu)(x-x_0)_T}}\nonumber\\
		&-2\intt{\intgg{\sigmazd\odot(\pi\partial_T z_T\pi+ z_\nu\partial_T\nu)\pi\partial_T(x-x_0)_T\pi}} -2\intt{\intgg{\partial_T z_\nu\left(\sigmazd(\partial_T\nu)(x-x_0)_T\right)}}\nonumber\\
		&+\intt{\intgg{\left(\sigmazd\odot\epsilonzd\right)\div_T(x-x_0)_T}}+ \intt{\intgg{|\nabla_T z_\nu|^2\div_T(x-x_0)_T}}\nonumber\\
		&-2\intt{\intgg{\left(\left( \Gamma_{\tau\kappa,\rho}^\mu-\Gamma_{\rho\kappa,\tau}^\mu+\Gamma_{\rho\varrho}^\mu\Gamma_{\tau\kappa}^\varrho-\Gamma_{\tau\varrho}^\mu\Gamma_{\rho\kappa}^\varrho\right)\odot\sigmazd\right) :\left( z_T\otimes (x-x_0)_T\right)}}\nonumber\\
		&+2\intt{\intgg{(\sigmazd\odot\partial_T\nu)( z^t_T\partial_T\nu)(x-x_0)_T}}+\intt{\intgg{h(x)| z^2|\div_T(x-x_0)_T}} \nonumber\\
		&-2\intt{\intgg{\gt\deltapt\left(\pi\partial_T z_T+ z_\nu\partial_T\nu\right)}} -2\intt{\intgg{u^\prime_T\left(\pi\partial_T z_T\pi+ z_\nu\partial_T\nu\right)(x-x_0)_T}}\nonumber\\ &-2\intt{\intgg{\rt z_T( z_\nu\partial_T\nu)(x-x_0)_T}}+2\intt{\intgg{\hnu z_\nu( z^t_T\partial_T\nu)(x-x_0)_T}}\nonumber\\
		& - \intt{\intgg{f_\nu(x)|\deltapnu|^2\div_T\ (x-x_0)_T}} -2\intt{\intgg{f_\nu(x)\deltapnu(z^t_T\partial_T\nu)(x-x_0)_T}} \nonumber\\
		&-2\intt{\intgg{u^\prime_\nu(\partial_T z_\nu- z^t_T\partial_T\nu)(x-x_0)_T}}-2\intt{\intgg{\gnu z^\prime_\nu(\partial_T z_\nu- (z_T)^t\partial_T\nu)(x-x_0)_T}}\nonumber\\
		& + \intt{\intgg{f(x)| z^\prime|^2}} +\intt{\intgg{h(x)| z|^2}} + \intt{\intgg{\sigmazd\odot\epsilonzd}}\nonumber\\ &-2\intt{\intgg{\nabla_T z_\nu^t\pi\partial_T(x-x_0)_T\pi\nabla_T z_\nu}}-2\intt{\intgg{\left(\nabla_T z_\nu\right)^t\nabla_T\cdot\left[z_T^t(\partial_T\nu)(x-x_0)_T\right]}} ,
	\end{align}	 
	wherein $\Gamma_{\lambda\eta}^\theta$ are the Christoffel symbols and $1\leq\kappa,\rho,\varrho,\mu\leq2$ (see for instance \cite{Hemina1999, KasriHemina2016}) .
\end{lemma}
\begin{proof}
	Multiplying both sides of the first equation in $\eqref{syst_PDE}$ by $2(q\odot\nabla u)$, integrating by parts over $\Omega\times [0;T]$ yields to
	\begin{align}\label{egalite_fond_equat_1}
		&-\left[2\intgo{u^\prime(q\odot\nabla u)}\right]_0^T +\intt{\intgo{\div q\left(\sigmau\odot\epsilonu-|u^\prime|^2\right)}} -2\intt{\intgo{\sigmau\odot\nabla q\cdot \nabla u}}\nonumber\\
		&\qquad -2\intt{\intgo{(a(x)u^\prime)(q\odot\nabla u)}} + \intt{\intgg{(q\cdot\nu)\left(|u^\prime|^2-\sigmau\odot\epsilonu\right)}}\nonumber\\
		&\qquad+2\intt{\intgg{ z^\prime(q\odot\nabla u)}}=0.
	\end{align}
	Next, we multiply the first acoustic boundary condition by $-2(\pi\partial_T z_T\pi+ z_\nu\partial_T\nu)q_T$, integrating over $\Gamma\times [0,T]$, we get	
	\begin{align}
		&-2\intt{\intgg{f_T(x) z_T^{\prime\prime}(\pi\partial_T z_T\pi+ z_\nu\partial_T\nu)q_T}}    +2\intt{\intgg{\div_T(\sigmazd)(\pi\partial_T z_T\pi+ z_\nu\partial_T\nu)q_T}}\nonumber\\
		&-2\intt{\intgg{g_T(x)  z_T^{\prime}(\pi\partial_T z_T\pi+ z_\nu\partial_T\nu)q_T}}   -2\intt{\intgg{r_T(x)  z_T(\pi\partial_T z_T\pi+ z_\nu\partial_T\nu)q_T}}\nonumber\\
		&-2\intt{\intgg{u_T^{\prime}(\pi\partial_T z_T\pi+ z_\nu\partial_T\nu)q_T }}= 0.
	\end{align}
	By applying integration by parts and utilizing Stokes formula \eqref{stokes_formula},
	we evaluate the resulting term $\sigmazd\odot\pi\partial_T\left[(\pi\partial_T z_T\pi+ z_\nu\partial_T\nu)q_T)\right]\pi$ as presented in \cite{Hemina1999}, which leads to	
	\begin{align}\label{egalite_fond_cond_bord_1}
		&-2\left[\intgg{\ft\deltapt(\pi\partial_T z_T\pi)q_T}\right]_0^T -2\left[\intgg{\ft\deltapt( z_\nu\partial_T\nu)q_T}\right]_0^T\nonumber\\
		&-\intt{\intgg{\ft\div_Tq_T|\deltapt|^2}} + \intt{\intgg{\ft\deltapt( z^\prime_\nu\partial_T\nu)q_T}}\nonumber\\
		&-2\intt{\intgg{\sigmazd\odot(\pi\partial_T z_T\pi+ z_\nu\partial_T\nu)\pi\partial_Tq_T\pi}} -2\intt{\intgg{\partial_T z_\nu\ \sigmazd(\partial_T\nu)q_T}}\nonumber\\
		&+\intt{\intgg{(\sigmazd\odot\epsilonzd)\div_Tq_T}} + 2\intt{\intgg{(\sigmazd\odot\partial_T\nu)(\partial_T z^\prime_\nu)q_T}} \nonumber\\
		&-2\intt{\intgg{\left(\left( \Gamma_{\tau\kappa,\rho}^\mu-\Gamma_{\rho\kappa,\tau}^\mu+\Gamma_{\rho\varrho}^\mu\Gamma_{\tau\kappa}^\varrho-\Gamma_{\tau\varrho}^\mu\Gamma_{\rho\kappa}^\varrho\right)\odot\sigmazd\right)\odot\left( z_T \otimes q_T\right)}} \nonumber\\ 
		&-2\intt{\intgg{\gt\deltapt(\pi\partial_T z_T\pi+ z_\nu\partial_T\nu)q_T}}      -2\intt{\intgg{u^\prime_T(\pi\partial_T z_T\pi+ z_\nu\partial_T\nu)q_T}}\nonumber\\
		&+\intt{\intgg{\rt\div_Tq_T| z_T|^2}} -2\intt{\intgg{\rt z_T( z_\nu(\partial_T\nu))q_T}}.
	\end{align}
	Therefore, we multiply the second acoustic boundary condition by $-2(\partial_T z_\nu - z^t_T\partial_T\nu)q_T$, perform series of integration by parts and apply \eqref{stokes_formula}, we obtain 
	\begin{align}\label{egalite_fond_cond_bord_2} 
		&-2\left[\intgg{\fnu\deltapnu\partial_T z_\nu q_T}\right]_0^T + 2\left[\intgg{\fnu\deltapnu( z^t_T\partial_T\nu)q_T}\right]_0^T\nonumber\\
		& -\intt{\intgg{\fnu| z_\nu^\prime|^2\div_Tq_T}} -2\intt{\intgg{\fnu\deltapnu(\overline{ z^\prime_T}\partial_T\nu)q_T}}\nonumber\\
		&-2\intt{\intgg{(\sigmazd\odot\partial_T\nu)(\partial_T z_\nu- z^t_T\partial_T\nu)q_T}} -2\intt{\intgg{\gnu z^\prime_\nu(\partial_T z_\nu-\overline{ z_T}\partial_T\nu)q_T}}\nonumber\\
		&+2\intt{\intgg{\hnu z_\nu(z_T^t\partial_T\nu)q_T}} + \intt{\intgg{\hnu| z_\nu|^2\div_Tq_T}} \nonumber\\
		&-2\intt{\intgg{u^\prime_\nu(\partial_T z_\nu-\overline{ z_T}\partial_T\nu)q_T}} +2\intt{\intgg{\nabla_T z_\nu\pi\partial_Tq_T\pi\nabla_T z_\nu}} \nonumber\\
		&+ \intt{\intgg{|\nabla_T z_\nu|^2\div_Tq_T}}  -2\intt{\intgg{(\nabla_T z_\nu)^t\nabla_T\left( z_T^t(\partial_T\nu)q_T\right)}}=0.
	\end{align}
	Combining equalities \eqref{egalite_fond_equat_1}, \eqref{egalite_fond_cond_bord_1} and \eqref{egalite_fond_cond_bord_2} we infer	
	\begin{align}\label{egalite_fond_I}
		&-\left[2\intgo{u^\prime(q\odot\nabla u)}\right]_0^T +\intt{\intgo{\div q\left(\sigmau\odot\epsilonu-|u^\prime|^2\right)}} -2\intt{\intgo{\sigmau\odot\nabla q\cdot \nabla u}}\nonumber\\
		& -2\intt{\intgo{(a(x)u^\prime)(q\odot\nabla u)}} + \intt{\intgg{(q\cdot\nu)\left(|u^\prime|^2-\sigmau\odot\epsilonu\right)}}
		+2\intt{\intgg{ z^\prime(q\odot\nabla u)}}\nonumber\\
		&-2\left[\intgg{\ft\deltapt(\pi\partial_T z_T\pi)q_T}\right]_0^T -2\left[\intgg{\ft\deltapt( z_\nu\partial_T\nu)q_T}\right]_0^T\nonumber\\
		&-\intt{\intgg{\ft\div_Tq_T|\deltapt|^2}} + \intt{\intgg{\ft\deltapt( z^\prime_\nu\partial_T\nu)q_T}}\nonumber\\
		&-2\intt{\intgg{\sigmazd\odot(\pi\partial_T z_T\pi+ z_\nu\partial_T\nu)\pi\partial_Tq_T\pi}} -2\intt{\intgg{\partial_T z_\nu\ \sigmazd(\partial_T\nu)q_T}}\nonumber\\
		&+\intt{\intgg{(\sigmazd\odot\epsilonzd)\div_Tq_T}} + 2\intt{\intgg{(\sigmazd\odot\partial_T\nu)(\partial_T z^\prime_\nu)q_T}}\nonumber\\ &-2\intt{\intgg{\left(\left( \Gamma_{\tau\kappa,\rho}^\mu-\Gamma_{\rho\kappa,\tau}^\mu+\Gamma_{\rho\varrho}^\mu\Gamma_{\tau\kappa}^\varrho-\Gamma_{\tau\varrho}^\mu\Gamma_{\rho\kappa}^\varrho\right)\odot\sigmazd\right)\odot\left( z_T \otimes q_T\right)}}\nonumber\\
		& -2\intt{\intgg{u^\prime_T(\pi\partial_T z_T\pi+ z_\nu\partial_T\nu)q_T}} -2\intt{\intgg{\gt\deltapt(\pi\partial_T z_T\pi+ z_\nu\partial_T\nu)q_T}}\nonumber\\
		&+\intt{\intgg{h(x)\div_Tq_T| z|^2}} -2\intt{\intgg{\rt z_T( z_\nu(\partial_T\nu))q_T}}\nonumber\\
		&+2\intt{\intgg{\hnu z_\nu( z_T^t\partial_T\nu)q_T}} -2\left[\intgg{\fnu\deltapnu\partial_T z_\nu q_T}\right]_0^T \nonumber\\
		&+ 2\left[\intgg{\fnu\deltapnu( z^t_T\partial_T\nu)q_T}\right]_0^T -\intt{\intgg{\fnu| z^\prime|^2\div_Tq_T}} \nonumber\\
		&-2\intt{\intgg{\fnu\deltapnu( (z^\prime_T)^t\partial_T\nu)q_T}} -2\intt{\intgg{(\sigmazd\odot\partial_T\nu)(\partial_T z_\nu- z_T^t\partial_T\nu)q_T}} \nonumber\\
		&-2\intt{\intgg{\gnu z^\prime_\nu(\partial_T z_\nu- z_T^t\partial_T\nu)q_T}} -2\intt{\intgg{u^\prime_\nu(\partial_T z_\nu-z_T^t\partial_T\nu)q_T}}
		\nonumber\\
		&+2\intt{\intgg{\nabla_T z_\nu\pi\partial_Tq_T\pi\nabla_T z_\nu}} +\intt{\intgg{|\nabla_T z_\nu|^2\div_Tq_T}}\nonumber\\ 
		&-2\intt{\intgg{(\nabla_T z_\nu)^t\nabla_T\left(z_T^t(\partial_T\nu)q_T\right)}}=0.
	\end{align}
	\noindent On the other hand, multiplying the first equation of $\eqref{syst_PDE}$ by $2\psi u$ where $\psi\in C^1(\overline{\Omega})$ and using the same arguments as in \eqref{egalite_fond_I} we obtain
	\begin{align}\label{egalite_fond_II}
		&2\left[\intgo{u^\prime\psi u}\right]_0^T -2\intt{\intgo{\psi|u^\prime|^2}} -2\intt{\intgg{ z^\prime\psi u}}\nonumber \\
		&+2\intt{\intgo{\sigmau\odot\epsilonu\ \psi}}+2\intt{\intgo{\sigmau\left((\nabla\psi)u\right)}}  +2\intt{\intgo{a(x)u^\prime\psi u}}.
	\end{align} 
	\noindent Taking $q=x-x_0$ and $\psi= 1$ as a choice in \eqref{egalite_fond_I} and \eqref{egalite_fond_II} respectively and using the definition of the energy \eqref{egalit_energy} leads us to the desired fundamental identity \eqref{egalite_fond_GLOB} of Lemma~\ref{lemme_identit_fondam_final_1}.
\end{proof}

\subsection{Energy Estimates}
\noindent In this subsection, we will estimate the terms on the right-hand side of $\eqref{egalite_fond_GLOB}$. We will use $c_i, C_i$ through this subsection, to denote generic positive constants. Note also that through all this part, we consider that $\displaystyle R=\sup_{x\in\Omega} |x-x_0|$, and that $(x-x_0)_T$, $\div_{T}(x-x_0)$, $(x-x_0)\cdot\nu$ and $\partial_T \nu$ are  bounded by a generic constant. For simplification purpose, this constant is automatically included into the constants $c_i$ or $C_i$, depending on the context in each integral estimation. 
\noindent We estimate our first term by applying Young's inequality for $\varepsilon>0$ and using \eqref{cond_def_a_damping}together with Lemmas~\ref{Lemma_estima_q_nabla_u_u} and \ref{lemma_energ_decrea}, which yields to
\begin{align*}
	\left|-2\intgo{u^\prime\left((x-x_0)\odot\nabla u+ u\right)}\right|&\le C_\varepsilon\intgo{|u^\prime|^2} +\varepsilon \intgo{\left[(x-x_0)\odot\nabla u+ u\right]^2},\nonumber\\
	&\le \left(\tilde{c}_1 +\varepsilon\xi\right) E(t),
\end{align*} 
then 
\begin{equation}\label{estimat_integ_1}
	\left|-2\left[\intgo{u^\prime\left((x-x_0)\odot\nabla u+ u\right)}\right]_0^T\right|\le (c_1+\varepsilon) E(0).
\end{equation}
\noindent By an analogous reasoning we obtain 
\begin{align}\label{estimat_integ_2}
	\left|-2\intt{\intgo{a(x)u^\prime\left((x-x_0)\cdot\nabla u+u\right)}}\right|\le c_2 E(0)+ {\varepsilon C_1\intt{E(t)}}.
\end{align}
\noindent On the other hand, since $(x-x_0)\odot\nabla u=(x-x_0)_T\odot\nabla_Tu +(x-x_0)_\nu\frac{\partial u}{\partial\nu}\cdot\nu$ on  $\Gamma$, we can estimate the term $2\intt{\intgg{\deltap\left((x-x_0)\odot\nabla u+u\right)}}$ accordingly
\begin{align*}
	2\intt{\intgg{\deltap\left((x-x_0)\odot\nabla u+u\right)}}&=  2\intt{\intgg{\deltap\left((x-x_0)\odot\nabla u\right)}} +   2\intt{\intgg{\deltap u}},\\
	&= 2\intt{\intgg{\deltap\left((x-x_0)_T\nabla_Tu +(x-x_0)_\nu\partial_\nu u\cdot\nu\right)}}\\ &\quad+  2\intt{\intgg{\deltap u}},
\end{align*}
applying Young's inequality alongside relations \eqref{poincar}, \eqref{egalit_energy} and \eqref{egalite_energy_decreas}, we infer  
\begin{align}\label{estimat_integ_3}
	2\intt{\intgg{\deltap\left((x-x_0)\odot\nabla u+u\right)}}&\le 3C_\varepsilon\intt{\intgg{| z^\prime|^2}} +\varepsilon\intt{\intgg{\sigmazu\odot\epsilonzu}}\nonumber\\
	&\quad +\varepsilon\intt{\intgg{|\partial_\nu u|^2}} + (\mathcal{C}_p^2\varepsilon)\intt{\intgo{\sigmau\odot\epsilonu}},\nonumber\\
	&\qquad\le c_3 E(0) +\varepsilon\intt{\intgg{\left[\sigmazu\odot\epsilonzu +|\partial_\nu u|^2\right]}}\nonumber\\
	&\quad \qquad+\varepsilon C_2\intt{E(t)}.   
\end{align}
\noindent Next, the term below is estimated through Young's inequality along with Lemma~ \ref{Lemme_equiva_norm} resulting in
\begin{align} \label{estimat_integ_4}
	&\left|-2\left[\intgg{\ft\deltapt(\pi\partial_T z_T\pi)(x-x_0)_T}\right]_0^T -2\left[\intgg{\ft\deltapt( z_\nu\partial_T\nu)(x-x_0)_T}\right]_0^T\right|\nonumber\\
	&\qquad\qquad\qquad\le c^*_4C_\varepsilon \left[\intgg{f(x)| z^\prime|^2}\right]_0^T + \varepsilon\left[\intgg{| z|^2} +\intgg{\sigmazd\odot\epsilonzd}\right]_0^T,\nonumber\\
	&\qquad\qquad\qquad\le c_4 E(0) +\varepsilon\left[\intgg{| z|^2} +\intgg{\sigmazd\odot\epsilonzd}\right]_0^T.
\end{align}
\noindent Moreover, by directly applying Young's inequality and using the energy decay identity \eqref{egalite_energy_decreas}, together with the boundedness of both $\div_T(x-x_0)_T$ and $\partial_T\nu$, we get
\begin{align}\label{estimat_integ_5}
	&\left|-\intt{\intgg{\ft\div_T(x-x_0)_T |\deltapt|^2}} +2\intt{\intgg{\ft\deltapt(\deltapnu\partial_T\nu)(x-x_0)_T}}\right|\nonumber\\
	&\hspace{3.cm}\le\frac{R\|f\|_{L^\infty}(1 +C_\varepsilon)}{g_0} \intt{\intgg{g(x)| z^\prime|^2}}+ \varepsilon \intt{\intgg{| z^\prime|^2}},\nonumber\\
	&\hspace{3.cm}\le c_5 E(0) + \varepsilon \intt{\intgg{| z^\prime|^2}}.
\end{align}
\noindent On the other hand, using the fact that $(x-x_0)_T$ and $\partial_T\nu$ are bounded, relations \eqref{defini_sigma0_T}, \eqref{defini_varepsilon0_T} and Lemma~\ref{Lemme_equiva_norm}, we obtain the following estimates  
\begin{equation}\label{estimat_integ_6}
	\left|-2\intt{\intgg{\sigmazd\odot(\pi\partial_T z_T\pi+ z_\nu\partial_T\nu)\pi\partial_T(x-x_0)_T\pi}}\right|\le C_4(\alpha, \lambda, R,\varepsilon)\intt{\intgg{\sigmazd\odot\epsilonzd}},
\end{equation}
\begin{align}
	\left|-2\intt{\intgg{\partial_T z_\nu\left(\sigmazd(\partial_T\nu)(x-x_0)_T\right)}}\right|&\le   C_5C_\varepsilon\intt{\intgg{\left(\sigmazd\odot\epsilonzd+| z|^2\right)}}\nonumber\\ &\hspace{1.cm} +\varepsilon\intt{\intgg{|\nabla_T z_\nu|^2}},\nonumber\\
	&\hspace{0.cm} \le C_5C_\varepsilon\intt{\intgg{\left(\sigmazd\odot\epsilonzd+| z|^2\right)}}\nonumber\\
	&\hspace{1.cm} + {\varepsilon C_6\intt{E(t)}}.\label{estimat_integ_7}
\end{align}
\begin{align}
	&\left|\intt{\intgg{\left(\sigmazd\odot\epsilonzd\right)\div_T(x-x_0)_T}}\right|+2\left|\intt{\intgg{(\sigmazd\odot\partial_T\nu)(z^t_T\partial_T\nu)(x-x_0)_T}}\right| \nonumber\\ 
	&+2\left|\intt{\intgg{\left(\left( \Gamma_{\tau\kappa,\rho}^\mu-\Gamma_{\rho\kappa,\tau}^\mu+\Gamma_{\rho\varrho}^\mu\Gamma_{\tau\kappa}^\varrho-\Gamma_{\tau\varrho}^\mu\Gamma_{\rho\kappa}^\varrho\right)\odot\sigmazd\right) \odot\left( z_T\otimes (x-x_0)_T\right)}}\right|\nonumber\\
	&+\left|\intt{\intgg{\sigmazd\odot\epsilonzd}}\right|\le C^*_7(R+ \varepsilon)\intt{\intgg{\sigmazd\odot\epsilonzd}} +C_\varepsilon\intt{\intgg{| z|^2}},\nonumber\\
	&\hspace{4.9cm}\le C_7\intt{\intgg{\left[\sigmazd\odot\epsilonzd +| z|^2\right]}}.\label{estimat_integ_8}
\end{align}
\noindent Thanks to assumption \eqref{cond_f_g_h_ge_f0_g0_h0}, the boundedness of $\partial_T\nu$ and same reasoning as above, we deduce
\begin{align}
	\left|-2\intt{\intgg{\gt\deltapt(\pi\partial_T z_T\pi+ z_\nu\partial_T\nu)(x-x_0)_T}}\right|&\le \varepsilon\intt{\intgg{\left(| z|^2 +\sigmazd\odot\epsilonzd\right)}}\nonumber\\ &\hspace{1.cm}+\frac{C_\varepsilon}{g_0}\intt{\intgg{g(x)| z^\prime|^2}},\nonumber \\
	&\le \varepsilon\intt{\intgg{| z|^2 +\sigmazd\odot\epsilonzd}}\nonumber\\
	&\qquad+c_6 E(0) .\label{estimat_integ_9}
\end{align}
\begin{align}\label{estimat_integ_10}
	&\left|\intt{\intgg{\rt\div_T(x-x_0)_T| z_T|^2}}+ \intt{\intgg{\hnu| z_\nu|^2\div_T(x-x_0)_T}}\right|\nonumber\\
	&\hspace{5.cm}\le C_9(\varepsilon +C_\varepsilon)\intt{\intgg{| z|^2}}.	
\end{align}
\begin{equation}\label{estimat_integ_11}
	\begin{split}
		&\left|-2\intt{\intgg{\rt z_T( z_\nu(\partial_T\nu))(x-x_0)_T}} +2\intt{\intgg{\hnu z_\nu( z_T^t\partial_T\nu)(x-x_0)_T}}\right|\nonumber\\
		&\hspace{+8.cm}\le C_{10}(\varepsilon +C_\varepsilon)\intt{\intgg{| z|^2}}.
	\end{split}
\end{equation}
\begin{align}
	&\left|-2\intt{\intgg{u^\prime_T\left(\pi\partial_T z_T\pi+ z_\nu\partial_T\nu\right)(x-x_0)_T}}\right|\le\left|-2\intt{\intgg{u^\prime_T\left(\pi\partial_T z_T\pi\right)(x-x_0)_T}}\right|\nonumber\\ 
	&\hspace{+8cm}+\left|2\intt{\intgg{u^\prime_T\left( z_\nu\partial_T\nu\right)(x-x_0)_T}}\right|,\nonumber\\
	&\hspace{3cm}\le  C_{11}C_\varepsilon\intt{\intgg{\left[| z|^2+ \sigmazd\odot\epsilonzd\right]}}+ {\varepsilon\intt{\intgg{|u^\prime|^2}}}.\label{estimat_integ_12}
\end{align}
Through Young's inequality and the boundedness of $f_\nu$, $(x-x_0)_T$ and $\partial_T\nu$, combined with Lemma~\ref{lemma_energ_decrea} we get
\begin{align}
	\left|-2\left[\intgg{\fnu\deltapnu\partial_T z_\nu (x-x_0)_T}\right]_0^T\right|\le& \left[\varepsilon c^*_{7}\intgg{|\nabla_T z_\nu|^2} +C_\varepsilon\|f\|_{L^\infty} R\intgg{f(x)|\deltap|^2}\right]_0^T,\nonumber\\
	\le & c_{7}E(0)  +\varepsilon \left[c^*_{7}\intgg{|\nabla_T z_\nu|^2}\right]_0^T, \label{estimat_integ_13}
\end{align} 
with $c_7=c_7(\varepsilon,R, \|f\|_{L^\infty})$.
\begin{align}
	+2\left[\intgg{f_\nu(x)\deltapnu(z^t_T\partial_T\nu)\ (x-x_0)_T}\right]_0^T\le & C_\varepsilon h_0^{-1}R\left[\intgg{h(x)| z|^2}\right]_0^T +\tilde{c}_8\varepsilon\left[\intgg{| z^\prime|^2}\right]_0^T,\nonumber\\
	\le&  c_8 E(0) + \tilde{c}_8\varepsilon\left[\intgg{| z^\prime|^2}\right]_0^T.\label{estimat_integ_14}
\end{align} 
\begin{align}	
	&\left|- \intt{\intgg{f_\nu(x)|\deltapnu|^2\div_T\ (x-x_0)_T}}-2\intt{\intgg{f_\nu(x)\deltapnu(z^t_T\partial_T\nu)(x-x_0)_t}}\right|\nonumber\\
	&\hspace{2.cm}\le g_0^{-1}\|f\|_{L^\infty} R(1+C_\varepsilon\|\partial_T\nu\|_{L^\infty}) \intt{\intgg{g(x)|\deltap|^2}} +\varepsilon c^*_{9}\intt{\intgg{| z|^2}},\nonumber\\
	&\hspace{2.cm}\le c_{9}E(0) + \varepsilon c^*_{9}\intt{\intgg{| z|^2}}.\label{estimat_integ_15}
\end{align} 
\noindent Following analogous reasoning, the next terms are managed as follows  
\begin{align}
	&\left|-2\intt{\intgg{u^\prime_\nu(\partial_T z_\nu- z^t_T\partial_T\nu)(x-x_0)_T}}\right|\le  R C_\varepsilon\intt{\intgg{\left(|\nabla_T z_\nu|^2 +h(x)| z|^2\right)}}\nonumber\\   &\hspace{8.cm}+2\varepsilon\intt{\intgg{|u^\prime|^2}},  \nonumber\\
	&\hspace{4cm}\le C_{12}\intt{\intgg{\left(|\nabla_T z_\nu|^2+ h(x)| z|^2\right)}}+ {2\varepsilon\intt{\intgg{|u^\prime|^2}}}.\label{estimat_integ_16}
\end{align}
\begin{align}
	&\left|-2\intt{\intgg{\gnu z^\prime_\nu(\partial_T z_\nu- (z_T)^t\partial_T\nu)(x-x_0)_T}}\right| \le c_{10} E(0){+\varepsilon C_{13}\intt{E(t)}}.\label{estimat_integ_17}
\end{align}
\begin{align}
	\intt{\intgg{f(x)| z^\prime|^2}} +\intt{\intgg{h(x)| z|^2}} \le c_{11}E(0) + C_{14}\intt{\intgg{h(x)| z|^2}}.\label{estimat_integ_18} 
\end{align}
\noindent We proceed as in \cite{KasriBale2025}, by means of Young's inequality and the fact that $\partial_T(x-x_0)_T$, $\div_T(x-x_0)_T$, $\partial_T\nu$  are bounded, we obtain
\begin{align}
	&\left|-2\intt{\intgg{\nabla_Tz^t_\nu\pi\partial_T(x-x_0)_T\pi\nabla_T z_\nu}}\right| + \left|-2\intt{\intgg{\left(\nabla_T z_\nu\right)^t\nabla_T\cdot\left[ z^t_T(\partial_T\nu)(x-x_0)_T\right]}}\right|\nonumber\\
	& + \left| \intt{\intgg{|\nabla_T z_\nu|^2\div_T(x-x_0)_T}}\right|\le C_{13} \intt{\intgg{|\nabla_T z_\nu|^2}} + { \varepsilon \intt{\intgg{| z|^2}}}, \label{estimat_integ_19}
\end{align}
\noindent such that $C_{13}= C_{13}(C_\varepsilon, R)$. 
\noindent Therefore, using Lemma~\ref{Lemma_estimat_sigmau_epsilonu_decomposit_domain} (Appendix), letting $\varepsilon$ tends toward $0$ and combining \eqref{estimat_integ_1}-\eqref{estimat_integ_19} yields to 
\begin{align}\label{47}
	2\intt{E(t)}&\le cE(0) +\overline{C}\intt{\intgg{\left[ \sigmazd\odot\epsilonzd+h(x)| z|^2+ |\nabla_T z_\nu|^2\right]}}\nonumber\\
	&\qquad  +R_{\xi\varepsilon}\intt{\intgwe{|u|^2}},
\end{align}
\noindent with 
\begin{align*}
	c=&c_{13}+ \sum_{i=1}^{11}c_i,\quad \\
	\overline{C}=&\max\left\{C_4+C_5C_\varepsilon+C_7+C_{11}C_\varepsilon; C_5C_\varepsilon+C_7+C_9C_\varepsilon+C_{10}C_\varepsilon+C_{11}C_\varepsilon+C_{12}C_\varepsilon; C_{12}C_\varepsilon; C_{14}\right\}.
\end{align*}
\noindent Hence, the last term of \eqref{47} is estimated as a result of Lemma~\ref{Lemma_estima_nabla_T_delta_nu}, provided that $\tau$ is chosen small enough so that $0<\tau < \frac{2}{\overline{C}}$, which leads to 
\begin{align}
	2\intt{E(t)}&\le cE(0)+R_{\xi\varepsilon}\intt{\intgwe{|u|^2}}+ \frac{\overline{C}\widehat{C}}{\tau}E(0) +\tau\overline{C}\intt{E(t)},\nonumber\\
	&\le C^\star E(0) +\overline{C}\tau\intt{E(t)} +R_{\xi\varepsilon}\intt{\intgwe{|u|^2}},\nonumber\\
	&\le C^\star E(0) +R_{\xi\varepsilon}\intt{\intgwe{|u|^2}},
\end{align}
with $C^\star=c+\frac{\overline{C}\widehat{C}}{\tau}$, which yields to 
\begin{align}
	2TE(T)\le 2\intt{E(t)}\le  C^\star E(0) +R_{\xi\varepsilon}\intt{\intgwe{|u|^2}}.
\end{align}
Consequently, as a result of the compactness-uniqueness alongside the Aubin-Lions theorem arguments, we can estimate the lower order term $\intt{\intgwe{|u|^2}}$. This is ensured by the works of \cite{Lasiecka1993,FrotaMedeirosVicente2014}, where a contradiction argument is developed. The reader would find detailed reasoning in the references above, which we mimic to ensure the results. We get then 
\begin{equation*}
	2TE(T)\le C^\star E(0) +C(T)\left(\intt{\intgo{a(x)|u'|^2}} + \intt{\intgg{g(x)| z^\prime|^2}}\right).
\end{equation*}
\noindent It follows from Lemma~\ref{lemma_energ_decrea} that 
\begin{equation*}
	E(T)\le M (E(0)-E(T)) ,
\end{equation*}
where $M=\frac{C^\star+C(E_0,T)}{2T-C(E_0,T)}$, which is equivalent to
\begin{equation*}
	E(T)\le \frac{M}{1+M} E(0).
\end{equation*}
Combined with the semigroup property, this yields to 
\begin{equation*}
	E(t)\le \mathcal{K}_1 e^{-\frac{t}{\mathcal{K}_2}}E(0),  \forall t\ge 0,
\end{equation*}
with $\mathcal{K}_1=\frac{M+1}{M}>1$ and $\mathcal{K}_2=\frac{T}{\ln\left(\mathcal{K}_1\right)} $. Hence, the proof of Theorem \ref{theorem_uniform_stability} is complete. $\square$
\begin{remark}{\ }
	\begin{enumerate}
		\item The results discussed in this work pertain to linear models (with linear damping). However, as demonstrated in the case of the wave equation with dynamic boundary conditions (see for instance \cite{Cav2009, FrotaVicente20152016}), similar stability results can be expected for nonlinear models, at least for semilinear elastodynamic systems with acoustic boundary conditions.
		\item Theorem \ref{theorem_uniform_stability} remains valid when we allow $\Omega$ to be a general star-shaped domain with respect to $x_0$.
	\end{enumerate}
\end{remark}
\section{Conclusion and open problems}
In this work, we studied a class of hyperbolic equations equipped with Dirichlet and acoustic boundary conditions as a model for the elastodynamic equation with localized damping. First, we proved that the system generated a $C_0$-semigroup of contractions on an appropriate Hilbert space. Then, under specific geometrical conditions, we established the exponential stability of the system via the multiplier method. In particular, we showed that the energy decays exponentially when the dissipative feedback $a(x)u'$ is supported only on a subset of $\Omega$, i.e., supp $a(x)\subset\omega$. However, several important questions remain open and deserve further investigation in the future. The results presented in this paper rely crucially on the geometric conditions (\ref{geom_cond1}) and (\ref{geom_cond2}), which constitutes a particular case of the geometric control condition (cf. \cite{Kliu}), and is used to derive the exponential decay rates of the system. Whether this geometric assumption can be relaxed while still achieving stability results remains an open question. Thus, it would be of significant interest to investigate whether the techniques developed in this work can be effectively combined with alternative approaches such as the trace method introduced by Lasiecka and Tataru \cite{Lasiecka1993} beyond the multiplier method employed here. Such an extension could broaden the applicability of the results, enabling the treatment of more general geometric configurations as well as semilinear terms. Additionally, a worthwhile direction for future research would be to analyze the stability properties of problem (\ref{syst_PDE}) in domains with geometric singularities of polygonal or polyhedral type (cf. \cite{Masrour}). Another interesting problem would be to examine the boundary stabilization of elastodynamic systems, where damping is imposed solely through acoustic boundary conditions and without internal dissipation.
\section{Appendix}
\begin{lemma}\label{Lemma_estimat_sigmau_epsilonu_decomposit_domain}
	Let $\Omega$ be a bounded, open, connected set in $\mathbb{R}^3$, having a boundary 
	$\Gamma = \partial \Omega$ of class $C^3$. Let $T > 0$ be large enough. Then, for every solution 
	$(u,  z)$ of \eqref{syst_PDE}, there exists positive constants $c_{13}$ and $R_{\xi\varepsilon}$ such that
	\begin{equation}\label{estimat_sigmau_epsilonu_decomposit_domain}
		\intt{\intgg{\left(\sigmau\odot\epsilonu-|u^\prime|^2\right)}}\le c_{13} E(0)+ R_{\xi\varepsilon}\intt{\int_{\omega_\varepsilon} |u|^2dx}
	\end{equation}
\end{lemma}
\begin{proof}
	The proof is based on the works of \cite{Lionstome1}, see also \cite{Vicente2016} and references therein. The idea consists on constructing a tabular neighborhood as in figure \ref{drawing2}, and through	Lemma~3.1 in \cite{Lionstome1} we ensures the existence of a vector field $k\in \left(W^{1,\infty}(\Omega)\right)^3$ with $k=\nu$ on $\Gamma$. Recall  $\eqref{egalite_fond_I}$ and let put $q=k$, which yields to 
	\begin{align}
		&	-\left[2\intgwed{u^\prime (k\odot\nabla u)}\right]_0^T + \intt{\intgwed{\div  k\left(\sigmau\odot\epsilonu -|u^\prime|^2\right)}} -2\intt{\intgg{\sigmau\odot(\nabla k\nabla u)}}\nonumber \\
		& -2\intt{\intgwed{a(x)u^\prime (k\cdot\nabla u)}} +\intt{\intgg{\left( |u^\prime|^2 -\sigmau\odot\epsilonu\right)}}\nonumber\\
		& +2\intt{\intgg{ z^\prime(k\odot\nabla u)}} =0,
	\end{align}
	thus \begin{align}\label{equality_above}
		\intt{\intgg{\left(\sigmau\odot\epsilonu-|u^\prime|^2\right)}}=&	-\left[2\intgwed{u^\prime (k\odot\nabla u)}\right]_0^T + \intt{\intgwed{\div k\left(\sigmau\odot\epsilonu -|u^\prime|^2\right)}}\nonumber\\
		& -2\intt{\intgg{\sigmau\odot(\nabla k\nabla u)}}
		-2\intt{\intgwed{a(x)u^\prime (k\odot\nabla u)}}\nonumber\\ 
		&+2\intt{\intgg{ z^\prime(k\odot\nabla u)}},
	\end{align}
	\noindent each term of \eqref{equality_above} is estimated similarly as before and we infer 
	\begin{equation}\label{first_sigmau_epsilonu}
		\intt{\intgg{\left(\sigmau\odot\epsilonu-|u^\prime|^2\right)}}\le c_{12} E(0) + C_{14} \intt{\intgwed{\sigmau\odot\epsilonu}}.
	\end{equation}
	\begin{figure}[!h]
		\begin{center}
			\includegraphics[scale=0.8]{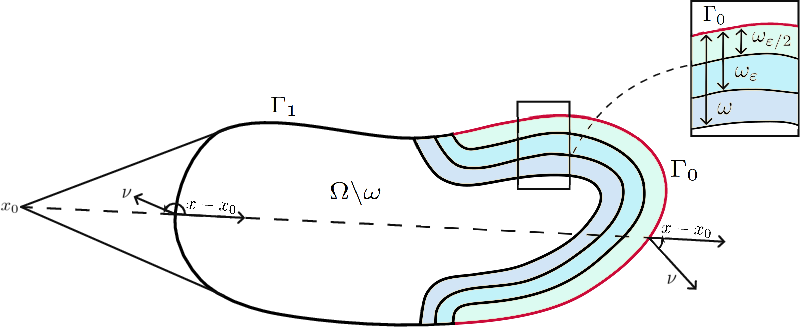}
			\captionof{figure}{ Illustation of collars, $\omega,\ \omega_\varepsilon,\ \omega_{\varepsilon/2}$, in the $\mathbb{R}^2$ case and an $x_0$ satisfying geometrical assumptions \eqref{geom_cond1}. \eqref{geom_cond2}} 
			\label{drawing2}
		\end{center}
	\end{figure}
	\noindent Now lets us consider a cutoff function $\xi_\varepsilon\in W^{1,\infty}(\Omega)$ (see \cite{Lionstome1} Lemma~2.4 Chap.VII) given by
	\begin{equation}\label{xi_varepsilon_definition}
		\xi_\varepsilon=\left\{ 
		\begin{array}{c c c}
			{1} &{ae} &{\omega_{\varepsilon/2}}\\
			{0\le \xi_\varepsilon\le 1} &ae &{\Omega}\\
			{0}& ae &{\Omega\backslash\omega}
		\end{array} 
		\right.
	\end{equation}
	such that $\displaystyle\frac{|\nabla \xi_\varepsilon|}{\xi_\varepsilon}\le R_{\xi\varepsilon}=\frac{R_\xi}{\varepsilon^2}$, where $R_\xi$ is a positive constant. Going back to \eqref{egalite_fond_II} and replacing $\psi$ by $\xi_\varepsilon$
	
	\begin{align}
		&2\left[\intgwe{u^\prime\xi_\varepsilon\ u}\right]_0^T    -2\intt{\intgwe{\xi_\varepsilon|u^\prime|^2}} -2\intt{\intgwe{ z^\prime\xi_\varepsilon u}}+ 2\intt{\intgwe{a(x)u^\prime\xi_\varepsilon u}}\nonumber\\
		&+2\intt{\intgwe{(\sigmau\odot\epsilonu)\xi_\varepsilon}} +2\intt{\intgwe{\sigmau\odot\left((\nabla\xi_\varepsilon)u\right)}} =0,
	\end{align}
	
	by rearranging the latter we obtain 
	
	\begin{align}
		2\intt{\intgwe{(\sigmau\odot\epsilonu)\xi_\varepsilon}}=& -2\left[\intgwe{u^\prime\xi_\varepsilon\ u}\right]_0^T  -2\intt{\intgwe{\sigmau\odot\left((\nabla\xi_\varepsilon)u\right)}}  \nonumber\\
		& +2\intt{\intgwe{\xi_\varepsilon|u^\prime|^2}}-2\intt{\intgwe{a(x)u^\prime\xi_\varepsilon u}}\nonumber\\
		&+2\intt{\intgwe{ z^\prime\xi_\varepsilon u}}.
	\end{align}
	
	\noindent  Thus, using the same arguments as before, each term in the previous expression is assessed separately, leading to the following 
	\begin{equation}
		\begin{aligned}
			\left|-2\intt{\intgwe{\sigmau\odot\left((\nabla\xi_\varepsilon)u\right)}}\right|& \le \intt{\intgwe{\xi_\varepsilon(\sigmau\odot\epsilonu)}} + \intt{\intgwe{\frac{|\nabla\xi_\varepsilon|^2}{\xi_\varepsilon}|u|^2}},\\
			&\le \intt{\intgwe{\xi_\varepsilon(\sigmau\odot\epsilonu)}} + R_{\xi\varepsilon}\intt{\intgwe{|u|^2}}.
		\end{aligned}
	\end{equation}
	\begin{equation}
		\begin{aligned}
			\left|2\intt{\intgwe{a(x)u^\prime\xi_\varepsilon u}}\right|&\le \|a(x)\|_{L^\infty} C_\varepsilon\intt{\intgwe{a(x)|u^\prime|^2}}  + \varepsilon\|\xi_\varepsilon\|_{L^\infty}^2\intt{\intgwe{|u|^2}}.
		\end{aligned}
	\end{equation}
	\begin{equation}
		\begin{aligned}
			2\intt{\intgwe{\xi_\varepsilon|u^\prime|^2}}\le \frac{\|\xi_\varepsilon\|_{L^\infty}}{a_0}\intt{\intgwe{a(x)|u^\prime|^2}}. 
		\end{aligned}
	\end{equation}
	\begin{equation}
		\begin{aligned}
			2\intt{\intgg{ z^\prime\xi_\varepsilon u}} &\le \frac{\|\xi_\varepsilon\|_{L^\infty}^2}{g_0}C_\varepsilon\intt{\intgg{g(x)| z^\prime|^2}} + \varepsilon\intt{\intgg{|u|^2}}, \\
			&\le \frac{\|\xi_\varepsilon\|_{L^\infty}^2}{g_0}C_\varepsilon E(0) +\varepsilon \intt{\intgg{|u|^2}}.
		\end{aligned}
	\end{equation}

	\begin{equation}
		\begin{aligned}
			2\left[\intgwe{u^\prime\xi_\varepsilon u}\right]_0^T &\le \|\xi_\varepsilon\|_{L^\infty} C_\varepsilon \left[\intgwe{|u^\prime|^2}\right]_0^T + \left[\varepsilon\intgwe{u^2}\right]_0^T,\\
			&\le \|\xi_\varepsilon\|C_\varepsilon E(0) +\varepsilon\left[\intgwe{u^2}\right]_0^T.
		\end{aligned}
	\end{equation}
	
	\noindent Combining inequalities above and choosing $\varepsilon$ small enough we get 
	
	\begin{equation}\label{estimat_sigmau_espilonu_xi_eps}
		\intt{\intgwe{(\sigmau\odot\epsilonu)\xi_\varepsilon}}\le c_{13} E(0) +R_{\xi\varepsilon}\intt{\intgwe{|u|^2}}, 
	\end{equation}
	
	\noindent where $c_{13}=c_{13}\left(\|\xi_\varepsilon\|_{L^\infty},C_\varepsilon,g_0^{-1},a_0^{-1}\right) $. Thus, taking into account \eqref{xi_varepsilon_definition} and the domain configuration as in Figure \ref{drawing2}, yields to  
	
	$$\int_{\omega_{\varepsilon/2}}\sigmau\odot\epsilonu\ dx=\int_{\omega_{\varepsilon/2}}(\sigmau\odot\epsilonu)\xi_\varepsilon\ dx\le \int_{\omega_{\varepsilon}}(\sigmau\odot\epsilonu)\xi_\varepsilon\ dx. $$
	
	\noindent Finally, injecting \eqref{estimat_sigmau_espilonu_xi_eps} in the latter estimation and therefor in \eqref{first_sigmau_epsilonu} leads to $\eqref{estimat_sigmau_epsilonu_decomposit_domain}$.  
\end{proof}

\end{document}